\documentclass{article}

\usepackage{arxiv}

\usepackage[utf8]{inputenc} 
\usepackage[T1]{fontenc}    
\usepackage{hyperref}       
\usepackage{url}            
\usepackage{booktabs}       
\usepackage{amsfonts}       
\usepackage{amssymb}
\usepackage{amsmath}
\usepackage{xcolor}
\usepackage{amsthm}
\usepackage{nicefrac}       
\usepackage{microtype}      
\usepackage{lipsum}		
\usepackage{graphicx}
\newtheorem{thm}{Theorem}

\newcommand{\thmref}[1]{Theorem~\ref{#1}}

\title{Construction of a family of $C^1$ convex integro cubic splines}

\author{ \href{https://orcid.org/0000-0000-0000-0000}{\includegraphics[scale=0.06]{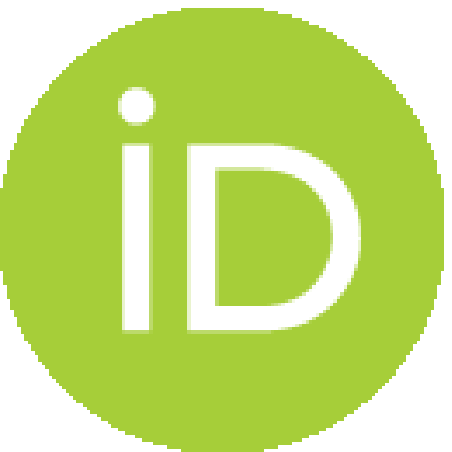}\hspace{1mm}Zhanlav,~T.}\\
	Institute of Mathematics and Digital Technologies\\
	Mongolian Academy of Sciences\\
	Ulaanbaatar, 13330 \\
	\texttt{tzhanlav@yahoo.com} \
	\And
	\href{https://orcid.org/0000-0002-4845-9019}{\includegraphics[scale=0.06]{eps/orcid}\hspace{1mm}Mijiddorj, R.} \\
	Department of Informatics\\
	Mongolian National University of Education\\
	Ulaanbaatar, 13330 \\
	\texttt{mijiddorj@msue.edu.mn} \\
}

\begin{document}
\maketitle

\begin{abstract}
We construct a family of monotone and convex $C^1$ integro cubic splines under a strictly convex position of the dataset. Then, we find an optimal spline by considering its approximation properties. Finally, we give some examples to illustrate the convex-preserving properties of these splines.
\end{abstract}

\keywords{Shape-preserving \and Approximation \and Integro spline}

\section{Introduction}
Let $\Delta_k$ be the non-uniform partition on $[a,b]$, $a=x_0<x_1<\cdots<x_k=b$, and $h_{i+1}=x_{i+1}-x_i$, $i=0,\cdots,k-1$, are step sizes. Let $S(x)$ be a cubic spline that approximates a function $u(x)$. We assume that the function values $u_i=u(x_i)$ are not given, but the integral values $h_{i+1}I_{i+1}$ on the subintervals $[x_i,x_{i+1}]$ of $u(x)$ are known. The problem of the construction of an integro cubic spline (see \cite{Beh06}) is to find $S(x)$ such that
\begin{equation}\label{zmcov1}
    \int\limits_{x_{i}}^{x_{i+1}}S(x)dx=\int\limits_{x_{i}}^{x_{i+1}}u(x)dx=h_{i+1}I_{i+1},~~~i=0,\cdots,k-1.
\end{equation}
We use a notation $m_i=S'(x_i)$ and piecewise polynomial representation
\begin{equation}\label{zmcov2}
\begin{split}
    S(x)&=(1-t)^2(1+2t)S(x_i)+t^2(3-2t)S(x_{i+1})+h_{i+1}t(1-t)\{(1-t)m_i-tm_{i+1}\},\\
    &x\in[x_i,x_{i+1}],~~t=\frac{x-x_i}{h_{i+1}},~~i=0,\cdots,k-1,~~t\in[0,1].
\end{split}
\end{equation}
The sufficient conditions for convexity of cubic histosplines derived in \cite{Neu78,Schm94} can be written in the form:
\begin{equation}\label{zmcov3}
 \begin{split}
   &2m_{i-1}+m_i\leq \frac3{h_i}(S(x_i)-S(x_{i-1}))\leq m_{i-1}+2m_i, ~~~i=1,\cdots,k,\\
   &\frac12(S(x_{i-1})+S(x_{i}))+\frac{h_i}{12}(m_{i-1}-m_i)=I_i, ~~~i=1,\cdots,k.
 \end{split}
\end{equation}
 Obviously, the system \eqref{zmcov3} has an infinite number of solutions. In \cite{Schm94,Schm3,Schm4}, the staircase algorithm with three terms of recurrence relations was used to find the solutions of \eqref{zmcov3}.
Moreover, shape-preserving approximations of histosplines have been studied in \cite{Fish2005,Fish2007,Zencak2002} and references therein. There are two traditional approaches to constructing shape-preserving histosplines: additional knots of spline and splines of a higher order with less smoothness \cite{Fish2007}. It is well known that the interpolating cubic spline of the class $C^2$ does not preserve the monotonicity and convexity of the input data. Recently, the shape-preserving properties of the $C^2$ local integro cubic spline have been investigated only on a uniform partition in \cite{Zhan2017}.\\
Usually, the monotonicity and convexity preserving property of the spline $S(x)$ are discussed based on the properties of the data $u_i=u(x_i)$. Now, we discuss the properties of monotonicity and convexity of the spline $S(x)$ based on data $I_i$. We construct a family of monotone and convex $C^1$ integro cubic splines under a strictly convex position of the data set.
In this paper, we will give a simple constructive algorithm for $C^1$ integro cubic splines (or histosplines) that preserve monotonicity and convexity.
The remainder of this paper is organized as follows. In section 2, a simple method for constructing the family of $C^1$ integro cubic splines (depending on the parameter $\alpha$) is given. We discuss sufficient conditions of monotonicity and convexity of the presented integro cubic splines. We also consider an error analysis of the integro cubic splines in Section 3. Some numerical examples are given in Section 4 to illustrate the convexity preserving property.
\section{Construction of convex integro cubic splines}
Using the ideas in \cite{Zhan1420,Zhan2135} instead of inequalities in \eqref{zmcov3}, we consider the following relations:
\begin{equation}\label{zmcov4}
 \begin{split}
   \frac3{h_i}(S(x_i)-S(x_{i-1}))&=\alpha(m_{i-1}+2m_i)+(1-\alpha)(2m_{i-1}+m_i)=\\
   &= (2-\alpha)m_{i-1}+(1+\alpha)m_i, ~~~\alpha\in [0,1].
 \end{split}
\end{equation}
The right-hand side of \eqref{zmcov4} is a linear combination of $m_{i-1}+2m_i$ and $2m_{i-1}+m_i$, and is a linear function with respect to $\alpha$. Hence, from \eqref{zmcov4}, it follows when $m_{i-1}\leq m_i$ that
\begin{equation}\label{zmcov5}
  2m_{i-1}+m_i\leq \frac3{h_i}(S(x_i)-S(x_{i-1}))\leq m_{i-1}+2m_i.
\end{equation}
That is, instead of \eqref{zmcov3}, it is possible to consider
\begin{equation}\label{zmcov6}
 \begin{split}
   &\frac3{h_i}(S(x_i)-S(x_{i-1}))=(2-\alpha)m_{i-1}+(1+\alpha)m_i,\\
   &\frac3{h_i}(S(x_i)+S(x_{i-1}))=\frac6{h_i}I_i-\frac12(m_{i-1}-m_i).
 \end{split}
\end{equation}
By adding and subtracting these two equations, we get
\begin{equation}\label{zmcov7}
  S(x_i)={I_i}+\frac{h_i}{12}\{(3-2\alpha)m_{i-1}+(3+2\alpha) m_i\}, ~~~i=1,\cdots,k,
\end{equation}
\begin{equation}\label{zmcov8}
  S(x_{i-1})=I_i+\frac{h_i}{12}\{(2\alpha-5)m_{i-1}-(2\alpha+1) m_i\}, ~~~i=1,\cdots,k.
\end{equation}
 From \eqref{zmcov7} and \eqref{zmcov8}, it follows that
\begin{equation}\label{zmcov9}
  \lambda_i(3-2\alpha)m_{i-1}+\{\lambda_i(3+2\alpha)+\mu_i(5-2\alpha)\}m_i+\mu_i(2\alpha+1)m_{i+1}=6\delta I_i,  ~~~i=1,\cdots,k-1,
\end{equation}
where
\begin{equation}\label{zmcov10}
  \lambda_i =\frac{h_i}{h_i+h_{i+1}}, ~~ \mu_i = 1 - \lambda_i,~~ \delta I_{i}=\frac{I_{i+1}-I_i}{\hbar_{i}},~~\hbar_i=\frac{h_i+h_{i+1}}{2}.
 \end{equation}
 Using the Eq. \eqref{zmcov7}, \eqref{zmcov8}, and \eqref{zmcov9}, we get a closed system of equations
\begin{equation}\label{zmcov11}
 \begin{split}
   &(5-2\alpha)m_{0}+(2\alpha+1)m_1=\frac{12}{h_1}(I_1-S(x_0)),\\
   &a_im_{i-1}+c_im_i+b_im_{i+1}=6\delta I_i,  ~~~i=1,\cdots,k-1,\\
   &(3-2\alpha)m_{k-1}+(3+2\alpha)m_k=\frac{12}{h_k}(S(x_k)-{I_k}),
 \end{split}
\end{equation}
where
\begin{equation}\label{zmcov12}
a_i=\lambda_i(3-2\alpha)>0,~~b_i=\mu_i(1+2\alpha)>0,~~~c_i=a_i+b_i+4(\alpha\lambda_i+(1-\alpha)\mu_i)>a_i+b_i>0.
\end{equation}
Since,
\begin{equation*}
 \begin{split}
   &5-2\alpha-2\alpha-1=4(1-\alpha)\geq0,\\
   &\lambda_i(3+2\alpha)+\mu_i(5-2\alpha)-\lambda_i(3-2\alpha)-\mu_i(2\alpha+1)=4\lambda_i\alpha+4\mu_i(1-\alpha)>0,  ~~~i=1,\cdots,k-1,\\
   &3+2\alpha-3+2\alpha=4\alpha\geq0,
 \end{split}
\end{equation*}
 the matrix of the system \eqref{zmcov11} has diagonal dominance. Hence, the system \eqref{zmcov11} has a unique solution\\ $(m_0,m_1,\cdots,m_k)$
for each $\alpha\in[0,1]$, and it can be easily solved \textcolor{black}{by} using the tridiagonal $LU$ decomposition algorithm. Here, $S(x_0)$ and $S(x_k)$ are assumed to be given for now. The values of $S(x_i)$ are determined by means of \eqref{zmcov7} or \eqref{zmcov8}, and the spline $S$ is given by \eqref{zmcov2}. Then, $S(x)$ will be $C^1$ cubic integro splines depending on the parameter $\alpha\in[0,1]$. Thus, the family of $S(x,\alpha)$ depending on the parameter $\alpha$ is determined completely. As usual, the given data $I_i$ is called monotonically increasing if
\begin{equation}\label{zmcov31}
\delta I_i=\frac{I_{i+1}-I_i}{\hbar_{i}} \geq 0,~~~i=1,\cdots,k-1,
\end{equation}
and convex if
\begin{subequations}\label{zmcov13}
\begin{equation}\label{zmcov13a}
\frac{I_{i+1}-I_i}{\hbar_{i}}-\frac{I_{i}-I_{i-1}}{\hbar_{i-1}}\geq0,~~~i=2,\cdots,k-1,
\end{equation}
or
\begin{equation}\label{zmcov13b}
\delta I_i \geq\delta I_{i-1}.
\end{equation}
\end{subequations}
Using the Taylor expansion of $S(x)$ in \eqref{zmcov2}, we obtain
\begin{equation*}S(x_0)=I_1+\frac{h_1}{12}\{\frac{\mu_1(1+2\alpha)(2\alpha-5)(\delta I_1-\delta I_2)}{\lambda_1(3-2\alpha)}-6\delta I_1\},
\end{equation*}
\begin{equation*}S(x_k)=I_k+\frac{h_k}{12}\{\frac{\lambda_{k-1}(9-4\alpha^2)(\delta I_{k-1}-\delta I_{k-2})}{\mu_{k-1}(1+2\alpha)}+6\delta I_{k-1}\}.
\end{equation*}
To study the shape-preserving properties of \eqref{zmcov2}, one must use the derivatives of \eqref{zmcov2}, which are
\begin{equation}\label{zmcov14}
    S'(x)=6t(1-t)\frac{S(x_{i+1})-S(x_i)}{h_{i+1}}+(1-t)(1-3t)m_i+t(3t-2)m_{i+1},
\end{equation}
 and
 \begin{equation}\label{zmcov15}
    S''(x)=6(1-2t)\frac{S(x_{i+1})-S(x_i)}{h^2_{i+1}}+\frac1{h_{i+1}}\{(6t-4)m_i+(6t-2)m_{i+1}\}.
\end{equation}
 Using \eqref{zmcov6} in \eqref{zmcov14} and \eqref{zmcov15}, we obtain
\begin{equation}\label{zmcov16}
    S'(x)=(1-t)(1+(1-2\alpha)t)m_{i}+t(2\alpha+t(1-2\alpha))m_{i+1},
\end{equation}
 and
 \begin{equation}\label{zmcov17}
    S''(x)=2\frac{\alpha+t(1-2\alpha)}{h_{i+1}}(m_{i+1}-m_{i}).
\end{equation}
It is easy to show that
\begin{equation*}
    1+(1-2\alpha)t\geq0,~~~2\alpha+t(1-2\alpha)\geq 0,~~~\mbox{for } \alpha\in[0,1].
\end{equation*}
Hence, from \eqref{zmcov16}, it follows that
\begin{equation}\label{zmcov18}
    S'(x)\geq0,~~x\in[x_i,x_{i+1}]~ \mbox{ if } m_i\geq0~ \mbox{ and } m_{i+1}\geq0.
\end{equation}
 Since $\alpha+t(1-2\alpha)\geq0$ then from \eqref{zmcov17}, it follows that
  \begin{equation}\label{zmcov19}
    S''(x)\geq0,~~x\in[x_i,x_{i+1}]~ \mbox{ if } m_{i+1}-m_i\geq0.
\end{equation}
Thus, from \eqref{zmcov18}, we conclude that $S(x,\alpha)$ will monotonically increase if the solution to \eqref{zmcov11} is nonnegative.
In order to study the solution to \eqref{zmcov11}, we use the following theorem given in \cite{Kay85}.
\begin{thm}
For the system $\textbf{Ax}=\textbf{f}$, suppose that
$$a_{ij}\geq0,~~ a_{ii}>0,~~f_i>0,~~i,j=1,\cdots,k,~~ i\neq j.$$
If for all $i$, $i = 1,\cdots,k,$
$$f_{i}>\sum\limits_{j=1,~ j\neq i}^{k}a_{ij}\frac{f_j}{a_{jj}},$$
then $\textbf{A}$ is invertible, and $x_i=(\textbf{A}^{-1}\textbf{f})_i> 0$ for all $i$.
\end{thm}
We show that the assumptions given in Theorem 1 are fulfilled for our system \eqref{zmcov11} under conditions
\begin{subequations}\label{zmc20}
\begin{equation}\label{zmc20a}
\frac {2a_1(I_1-S(x_0))}{h_1(5-2\alpha)}+\frac {b_1\delta I_2}{c_2}< \delta I_1<\frac {2c_1(I_1-S(x_0))}{h_1(2\alpha+1)},~~I_1-S(x_0)>0,
\end{equation}
\begin{equation}\label{zmc20b}
\frac {a_j\delta I_{j-1}}{c_{j-1}}+\frac {b_j\delta I_{j+1}}{c_{j+1}}< \delta I_j,~~~j=2,\cdots,k-2,
\end{equation}
\begin{equation}\label{zmc20c}
\frac {2b_{k-1}(S(x_k)-I_k)}{h_k(3+2\alpha)}+\frac {a_{k-1}\delta I_{k-2}}{c_{k-2}}< \delta I_{k-1}<\frac {2c_{k-1}(S(x_k)-I_k)}{h_k(3-2\alpha)},~~S(x_k)-I_k>0.
\end{equation}
\end{subequations}
Let us summarize the obtained above results as:
\begin{thm}
Let the integro cubic splines $S(x,\alpha)\in C^1[a,b]$ be defined by \eqref{zmcov2}, \eqref{zmcov7}, and \eqref{zmcov11}, and the data $I_i$ monotonically increase. If the inequalities \eqref{zmc20} are valid then $m_i>0$ for all $i=0,\cdots,k$ and thereby $S'(x)>0$ on $[x_0,x_{k}]$, that is, $S$ is monotonically increasing on $[a,b]$.
\end{thm}
Now, we proceed to study the convexity property of $S(x,\alpha)$. To this end, we pass from \eqref{zmcov11} to the following system
\begin{subequations}\label{zmc222}
 \begin{equation}\label{zmc22a}
   (2\alpha+1)(m_1-m_0)=\frac{12}{h_1}(I_1-S(x_0))-6m_0,
 \end{equation}
\begin{equation}\label{zmc22b}
 \begin{split}
   a_i(m_{i-1}-m_{i-2})&+c_i(m_{i}-m_{i-1})+b_i(m_{i+1}-m_{i})=6(\delta I_i-\delta I_{i-1})+\\
   &(a_{i-1}-a_{i})m_{i-2}+(c_{i-1}-c_{i})m_{i-1}+(b_{i-1}-b_{i})m_{i},  ~~~i=2,\cdots,k-1,
 \end{split}
 \end{equation}
\begin{equation}\label{zmc22c}
   (3+2\alpha)(m_k-m_{k-1})=\frac{12}{h_k}(S(x_k)-I_k)-6m_{k-1}.
 \end{equation}
 \end{subequations}
If the following equalities hold:
\begin{equation}\label{zmcov23''}
  a_{i-1}-a_{i}=0,~~c_{i-1}-c_{i}=0,~~b_{i-1}-b_{i}=0,
 \end{equation}
 then the equation \eqref{zmc22b} for $i=2,\cdots,k-1$ leads to
\begin{equation}\label{zmcov11_2}
   a_i(m_{i-1}-m_{i-2})+c_i(m_{i}-m_{i-1})+b_i(m_{i+1}-m_{i})=6(\delta I_i-\delta I_{i-1}),~~~i=2,\cdots,k-1.
 \end{equation}
As above, it is easy to verify that the assumptions of Theorem 1 are fulfilled for the system \eqref{zmc22a}, \eqref{zmc22c}, and \eqref{zmcov11_2} under conditions
\begin{subequations}\label{zmc20_1}
\begin{equation}\label{zmc20_1a}
\frac {a_2(\frac{2(I_1-S(x_0))}{h_1}-m_0)}{2\alpha+1}+\frac {b_2(\delta I_3-\delta I_2)}{c_3}< \delta I_2-\delta I_1,
\end{equation}
\begin{equation}\label{zmc20_1b}
\frac {a_j(\delta I_{j-1}-\delta I_{j-2})}{c_{j-1}}+\frac {b_j(\delta I_{j+1}-\delta I_{j})}{c_{j+1}}< \delta I_j-\delta I_{j-1},~~~j=3,\cdots,k-2,
\end{equation}
\begin{equation}\label{zmc20_1c}
\frac {b_{k-1}(\frac{2}{h_k}(S(x_k)-I_k)-m_{k-1})}{3+2\alpha}+\frac {a_{k-1}(\delta I_{k-2}-\delta I_{k-3})}{c_{k-2}}< \delta I_{k-1}-\delta I_{k-2},
\end{equation}
\end{subequations}
where $\frac{2}{h_1}(I_1-S(x_0))-m_0>0$ and $\frac{2}{h_k}(S(x_k)-I_k)-m_{k-1}>0$. Thus, we have:
\begin{thm}
Let the integro cubic splines $S(x,\alpha)\in C^1[a,b]$ be defined by \eqref{zmcov2}, \eqref{zmcov7}, and \eqref{zmcov11}, and the data $I_i$ are convex, and $m_0$ and $m_{k-1}$ are given. If \eqref{zmcov23''} and \eqref{zmc20_1} are valid then $m_i-m_{i-1}>0$ for all $i=1,\cdots,k$ and thereby $S''(x)>0$ on $[x_0,x_{k}]$, that is, $S(x)$ is  convex on $[a,b]$.
\end{thm}
\noindent Note that the equalities \eqref{zmcov23''} hold true if the step sizes of grid satisfy
\begin{equation}\label{zmcov26''}
h_{i}=\sqrt{h_{i-1}h_{i+1}},~~~i=1,\cdots,k-1.
\end{equation}
Of course, the conditions \eqref{zmcov26''} are fulfilled on a uniform partition.
Now, we are interested in the dependence of $m_i$ on parameter $\alpha$. To this end, differentiating the system \eqref{zmcov11} with respect to $\alpha$, we obtain
\begin{equation}\label{zmcov33}
 \begin{split}
   &(5-2\alpha)m'_{0}(\alpha)+(2\alpha+1)m'_1(\alpha)=2(m_0-m_1),\\
   &a_im'_{i-1}(\alpha)+c_im'_i(\alpha)+b_im'_{i+1}(\alpha)=2\lambda_i(m_{i-1}-m_i)+2\mu_i(m_i-m_{i+1}),  ~~~i=1,\cdots,k-1,\\
   &(3-2\alpha)m'_{k-1}(\alpha)+(3+2\alpha)m'_k(\alpha)=2(m_{k-1}-m_k).
 \end{split}
\end{equation}
From \eqref{zmcov33}, it is clear that the right-hand side of the system \eqref{zmcov33} is negative if \eqref{zmcov23''} and \eqref{zmc20_1} are fulfilled. 
\begin{thm}
Assume that \eqref{zmcov23''} and \eqref{zmc20_1} are fulfilled. Then, $m_i(\alpha)$ is a decreasing function with respect to $\alpha$ for all $i=1,\cdots,k-1$, i.e.,
\begin{equation}\label{zmcov38}
m_{i}(0)\geq m_{i}(\alpha)\geq m_{i}(1),~~~i=1,\cdots,k-1.
\end{equation}
\end{thm}
\begin{proof}
By Theorems 1 and 3, the solution  to system \eqref{zmcov33} is negative, that is, $m'_i(\alpha)<0$ for all $i=1,\cdots,k-1$. Hence, \eqref{zmcov38} is valid.
\end{proof}
Thus, we obtain feasible intervals $[m_i(1),m_i(0)]$ of $m_i(\alpha)$ that ensure the monotonicity and convexity of splines $S(x,\alpha)$. From \eqref{zmcov7} and \eqref{zmcov38}, we also derive the interval of $S(x_i,\alpha)$:
\begin{equation}\label{zmcov39}
S(x_i,\alpha)\in [I_i+\frac{h_i}2 m_{i-1}(1),I_i+ \frac{h_i}2m_{i}(0)],~~~i=1(1)k-1.
\end{equation}
\begin{thm}
Let the assumptions of Theorem 3 be fulfilled and $m_0>0$. Then, $S(x_i)$ are in strictly convex positions as the data $I_i$, that is,
\begin{equation}\label{zmcov40}
\delta S(x_i)>\delta S(x_{i-1})\geq 0,~~~i=1,\cdots,k-1.
\end{equation}
\end{thm}
\begin{proof}
By \eqref{zmcov4} and Theorem 3, we have
\begin{equation*}
\delta S(x_i)=\frac{S(x_{i+1})-S(x_i)}{h_{i+1}}=\frac13 \{(2-\alpha)m_i+(1+\alpha)m_{i+1}\}\geq m_i\geq0.
\end{equation*}
By Theorem 3, we have
\begin{equation*}
m_{i+1}- m_i>0,~~~i=0,\cdots,k-1,
\end{equation*}
\begin{equation*}
m_{i}- m_{i-1}>0,~~~i=1,\cdots,k.
\end{equation*}
 From this, we obtain
 \begin{equation*}
(1+\alpha)(m_{i+1}- m_i)+(2-\alpha)(m_i-m_{i-1})>0,
\end{equation*}
 which leads to
 \begin{equation}\label{zmcov41}
(1+\alpha)m_{i+1}+(2-\alpha)m_i>(1+\alpha)m_{i}+(2-\alpha)m_{i-1},~~~i=1,\cdots,k-1.
\end{equation}
Using \eqref{zmcov6} in \eqref{zmcov41}, we get \eqref{zmcov40}.
\end{proof}
The well-known convex interval interpolation problem was solved by three-term staircase algorithm in \cite{Mul1999}. From \eqref{zmcov13} and \eqref{zmcov39}, one can easily see that we solved the convex interval interpolation problem $S(x_i)\in [l_i,\upsilon_i]$, $i=0,\cdots,k,$ for a particular case with $l_i=I_i+\frac{h_i}2m_{i-1}(1),~~\upsilon_i=I_i+\frac{h_i}2m_{i}(0)$.

\section{Error analysis}
Now, we consider the approximation properties of convex integro cubic splines $S(x,\alpha)$. Using the Taylor expansion of $u\in C^3[a,b]$, one can easily obtain
\begin{equation}\label{zmcov44}
\delta I_{i}=u'_i+\frac{h_{i+1}-h_i}{3}u''_i+O(\bar{h}^2),
\end{equation}
\begin{equation}\label{zmcov45}
I_{i}=u_i-\frac{h_{i}}{2}u'_i+\frac{h^2_{i}}{6}u''_i+O(\bar{h}^3),
\end{equation}
where $\bar{h}=\max\limits_{1\leq i \leq k}h_i$.\\
From \eqref{zmcov44}, it is clear that
\begin{equation}\label{zmcov46}
\delta I_{i}=u'_i+O(\bar{h}^2),
\end{equation}
under condition
\begin{equation}\label{zmcov47}
h_{i+1}-h_i=O(\bar{h}^2).
\end{equation}
\begin{thm}\label{zmthm6}
Let $S(x,\alpha)$ be $C^1$ integro cubic splines defined by \eqref{zmcov2}, \eqref{zmcov7}, \eqref{zmcov11}, and $S(x_i)=u_i+O(\bar{h}^3),$ $i=0,k$. Then, for $u\in C^3$, we have estimations
\begin{equation}\label{zmcov48}
S^{(r)}(x_i)-u^{(r)}_i=O(\bar{h}^{\sigma+1-r}),~~r=0,1,~~i=0,\cdots,k.
\end{equation}
under \eqref{zmcov47}. Here $\sigma=1$ when $\alpha\neq \frac12$ and $\sigma=2$ when $\alpha=\frac12$.
\end{thm}
\begin{proof}
First, let us estimate $q_i=m_i-u'_i$, $i=0,\cdots,k$. To this end, we pass from \eqref{zmcov11} to the system
\begin{equation}\label{zmcov48'}
 \begin{split}
   &(5-2\alpha)q_{0}+(2\alpha+1)q_1=d_0,\\
   &a_iq_{i-1}+c_iq_i+b_iq_{i+1}=d_i, ~~i=1,\cdots,k-1,\\
   &(3-2\alpha)q_{k-1}+(3+2\alpha)q_k=d_k,
 \end{split}
\end{equation}
where
\begin{equation}\label{zmcov49}
 \begin{split}
   &d_0=\frac{12}{h_1}(I_1-S(x_0))-\{ (5-2\alpha)u'_{0}+(2\alpha+1)u'_1\},\\
   &d_i=6\delta I_i-\{a_iu'_{i-1}+c_iu'_i+b_iu'_{i+1}\}),\\
   &d_k=\frac{12}{h_k}(S(x_k)-I_k)-\{ (3-2\alpha)u'_{k-1}+(3+2\alpha)u'_k\}.
 \end{split}
\end{equation}
Using \eqref{zmcov45}, \eqref{zmcov46}, \eqref{zmcov47}, and the Taylor expansion of function $u\in C^3[a,b]$ in \eqref{zmcov49}, one can easily obtain \begin{equation}\label{zmcov50}
d_i=O(\bar{h}^{\sigma}).
\end{equation}
Then, from \eqref{zmcov48'} and \eqref{zmcov50}, it follows \eqref{zmcov48} for $r=1$. From \eqref{zmcov7}, we get
\begin{equation}\label{zmcov51}
 \begin{split}
S(x_i)-u_i&=I_i-u_i+\frac{h_i}{12}\big\{(3-2\alpha)(m_{i-1}-u'_{i-1})+(3+2\alpha)(m_i-u'_i)\big\}\\
   &+\frac{h_i}{12}\big\{(3-2\alpha)u'_{i-1}+(3+2\alpha)u'_i\big\},~~i=1,\cdots,k-1.
 \end{split}
\end{equation}
As above, using \eqref{zmcov45}, \eqref{zmcov48} for $r=1$ and the Taylor expansion of function $u\in C^3[a,b]$ in \eqref{zmcov51}, we obtain
\begin{equation*}
S(x_i)-u_i=O(\bar{h}^{\sigma+1}),~~~i=1,\cdots,k-1,
\end{equation*}
i.e., the estimate \eqref{zmcov48} is proved for $r=0$. This completes the proof of Theorem 6.
\end{proof}
Using \eqref{zmcov46} and \eqref{zmcov48} for $r=1$, it is easy to show that
\begin{equation}\label{zmcov52}
S''_{i+0}-S''_{i-0}=O(\bar{h}^{\sigma-1}),~~~i=1,\cdots,k-1.
\end{equation}
From the estimations \eqref{zmcov48} and \eqref{zmcov52}, it is clear that the best or optimal $C^1$ integro cubic spline (abbr. OCICS) is derived when $\alpha=\frac12$ in the sense of approximation properties. This selection shows that using an optimal choice of parameter one can rise the order of approximation.
\section{Numerical experiments}
\noindent In this section, we apply the proposed method to some numerical examples.

Example 1. We consider the histogram $I=\{1,2,4\}$ on $\Delta_3 =\{0<4<6<7\}$ in \cite{Schm94}. A convex integro cubic spline (abbr. CICS) curve with $\alpha=0.5$ is shown in Fig. 1, and with $\alpha=1$ is shown in Fig. 2. 

Example 2. Next, we take $u(x) = 2-\sqrt{x(2 -x)}$, $0\leq x \leq 2$ \cite{Kva12}. This function is approximated by the CICS on a uniform mesh in $x$, for $k=10$, in Fig. 3. In Fig. 4, we consider the CICS for this function on a non-uniform grid\\
 $\Delta_{10} =\{0<0.05<0.1<0.4<0.7<1<1.3<1.6<1.9<1.95<2\}$. Near the end knots, the fitting result of the spline curve in Fig. 4 is better than that of the spline curve in Fig. 3. From this example, we can see that the constructed CICS possesses convexity-preserving property and convergence. The purpose of this example is to observe the effects of the changes in the step size.

Example 3. Then, we consider the histogram $I=\{2.86,1,0.5,1,2,2.86\}$ on\\ $\Delta_6 =\{0<1<2<4<6<7<8\}$ which is in convex positions. Fig. 5 shows that the fitting result is the same as \textcolor{black}{that} presented in \cite{Schm94}.

Fortunately, for the data of the examples above, the conditions \eqref{zmc20_1} are fulfilled.

Table 1.
Akima's data \cite{Aki70}.\\
\begin{tabular}{ccccccccccccccccccccc}
  \hline
$x_i$&0&&2&&3&&5&&6&&8&&9&&11&&12&&14\\
\hline
$I_i$&&10&&10&&10&&10&&10&&10&&10.5&&15&&50&\\  \hline
\end{tabular}\\

\noindent
 \includegraphics[width=7cm]{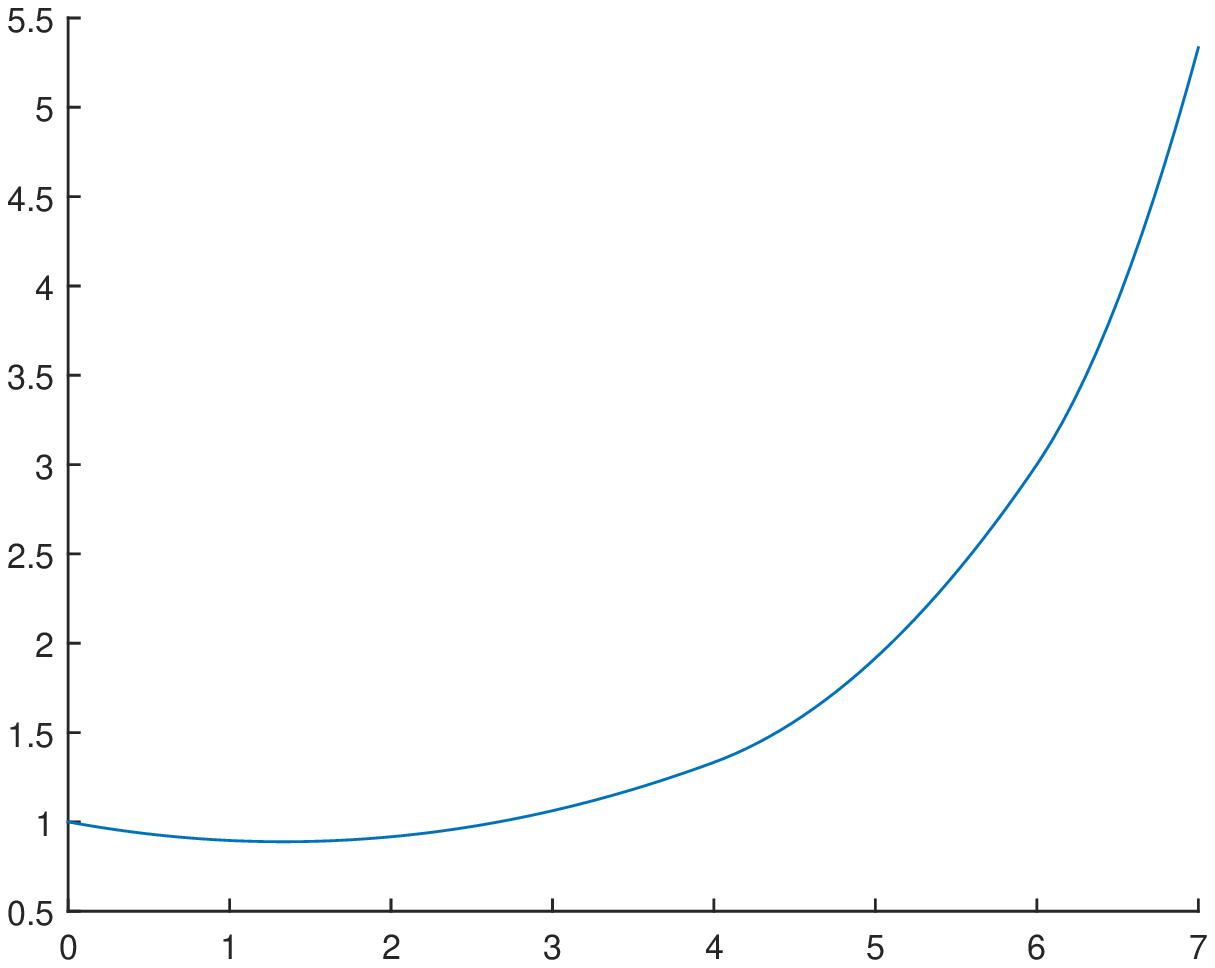} \hspace{2cm} 
 \includegraphics[width=7cm]{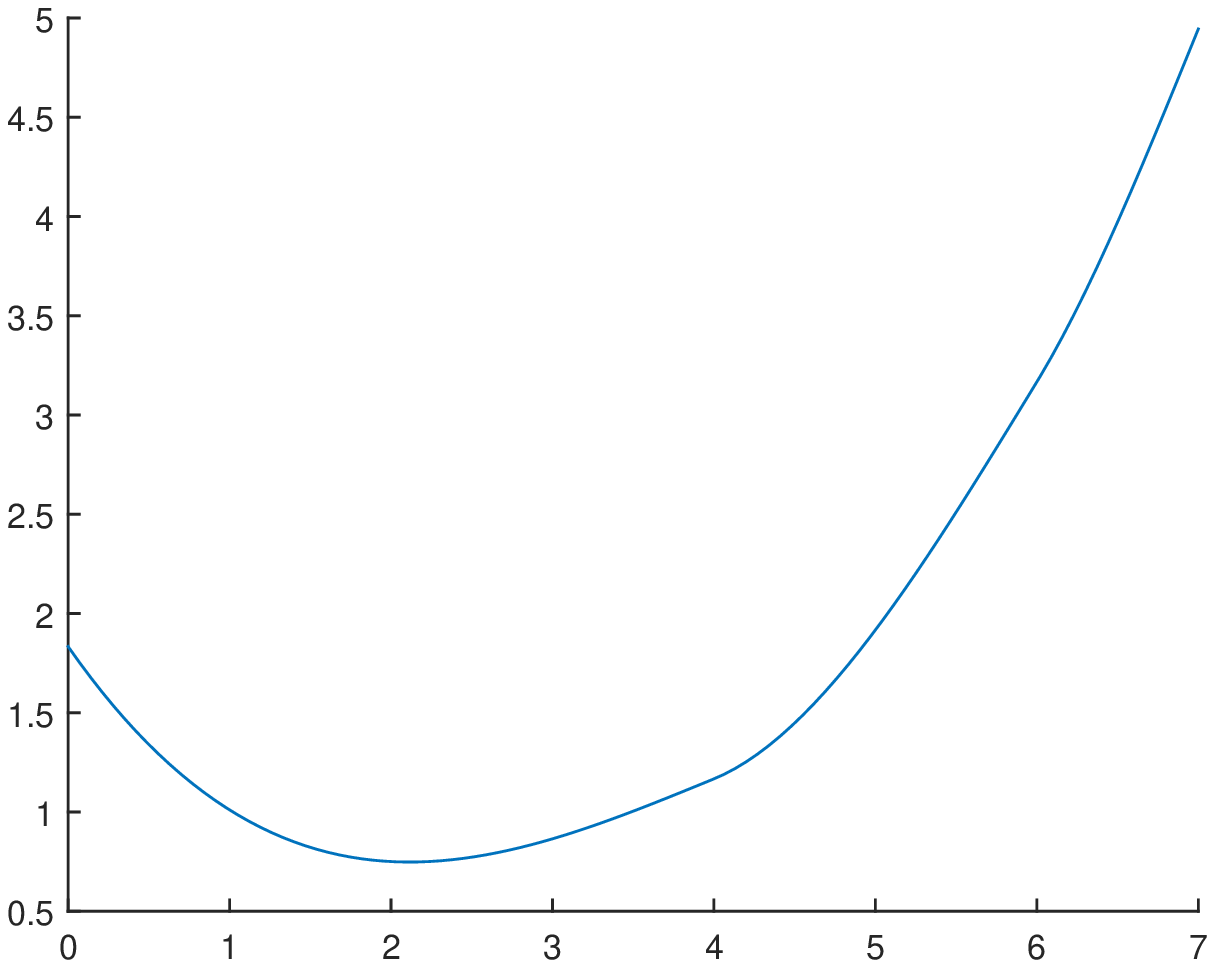} \hspace{2cm} 

{\scriptsize \noindent \textbf{Fig. 1.} Approximation by OCICS with $\alpha=0.5$ for Example 1.\hspace{0.5cm} \textbf{Fig. 2.} Approximation by CICS with $\alpha=1.0$ for Example 1.}\\
\includegraphics[width=7cm]{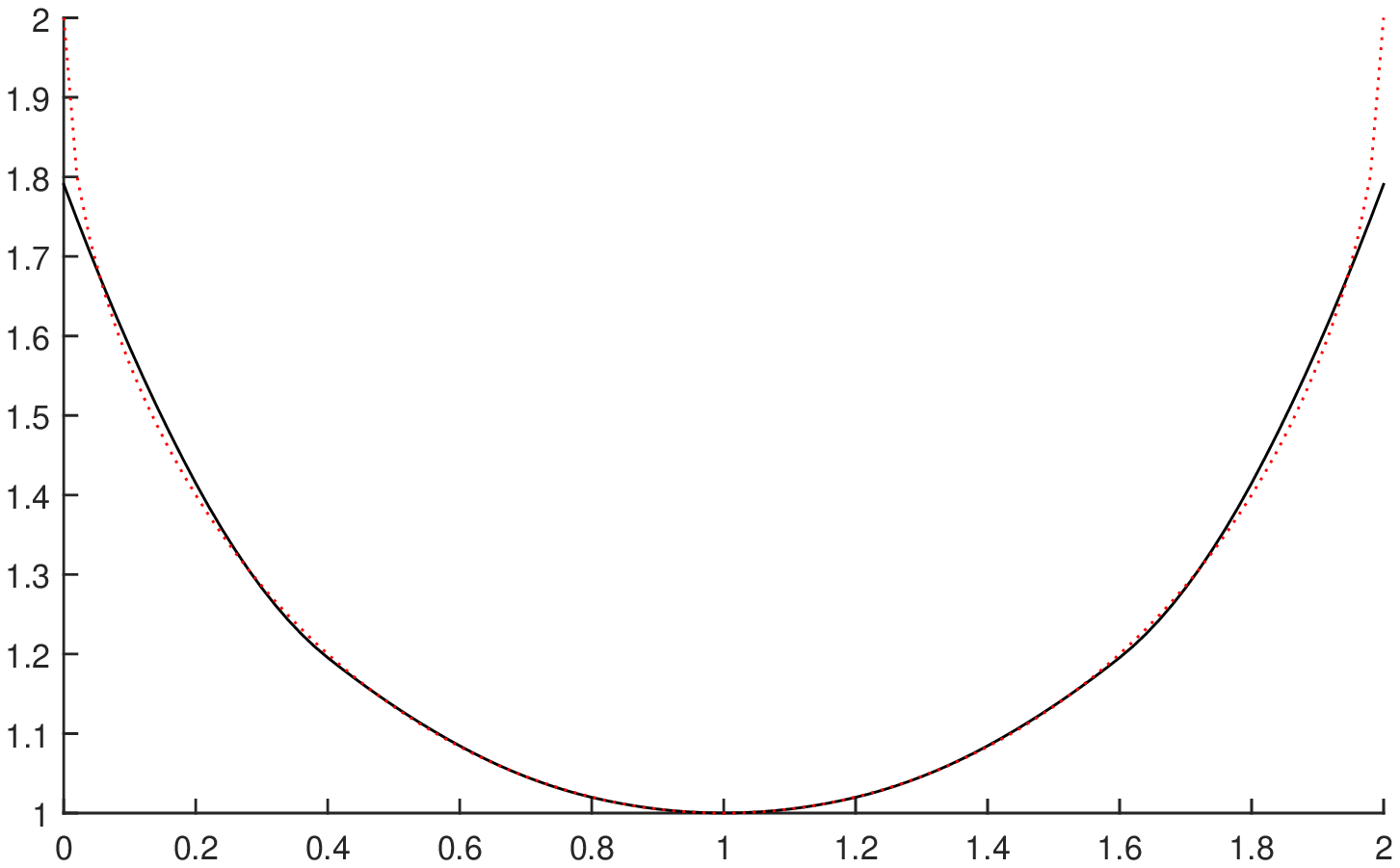} \hspace{2cm}
 \includegraphics[width=7cm]{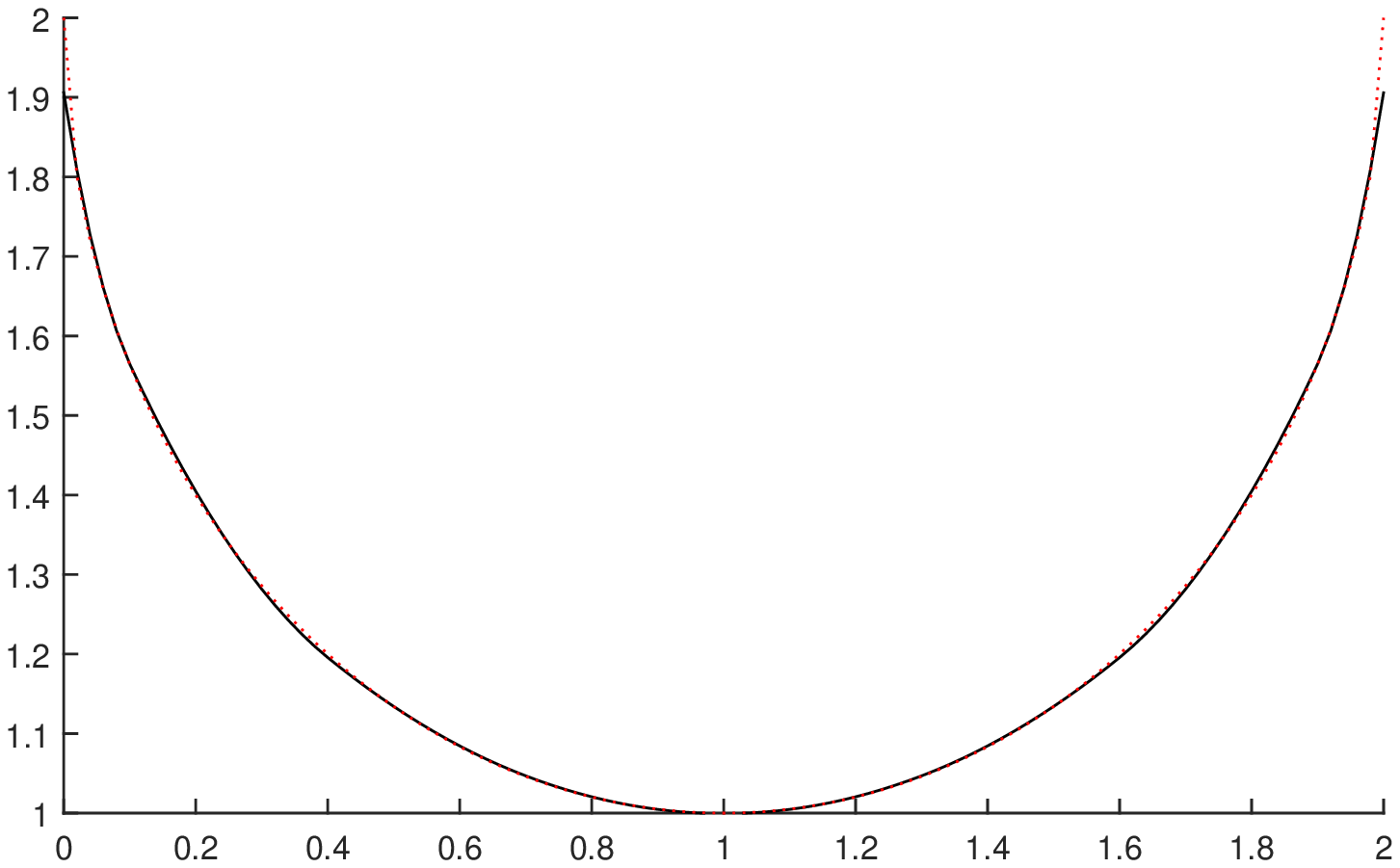} \hspace{2cm}\\
{\footnotesize \textbf{Fig. 3.} {Approximation by OCICS for $u(x),~k=10$.} \hspace{0.5cm} \textbf{Fig. 4.} Approximation by OCICS for $u(x)$ on $\Delta_{10}$.}\\
\noindent
 \includegraphics[width=6.5cm]{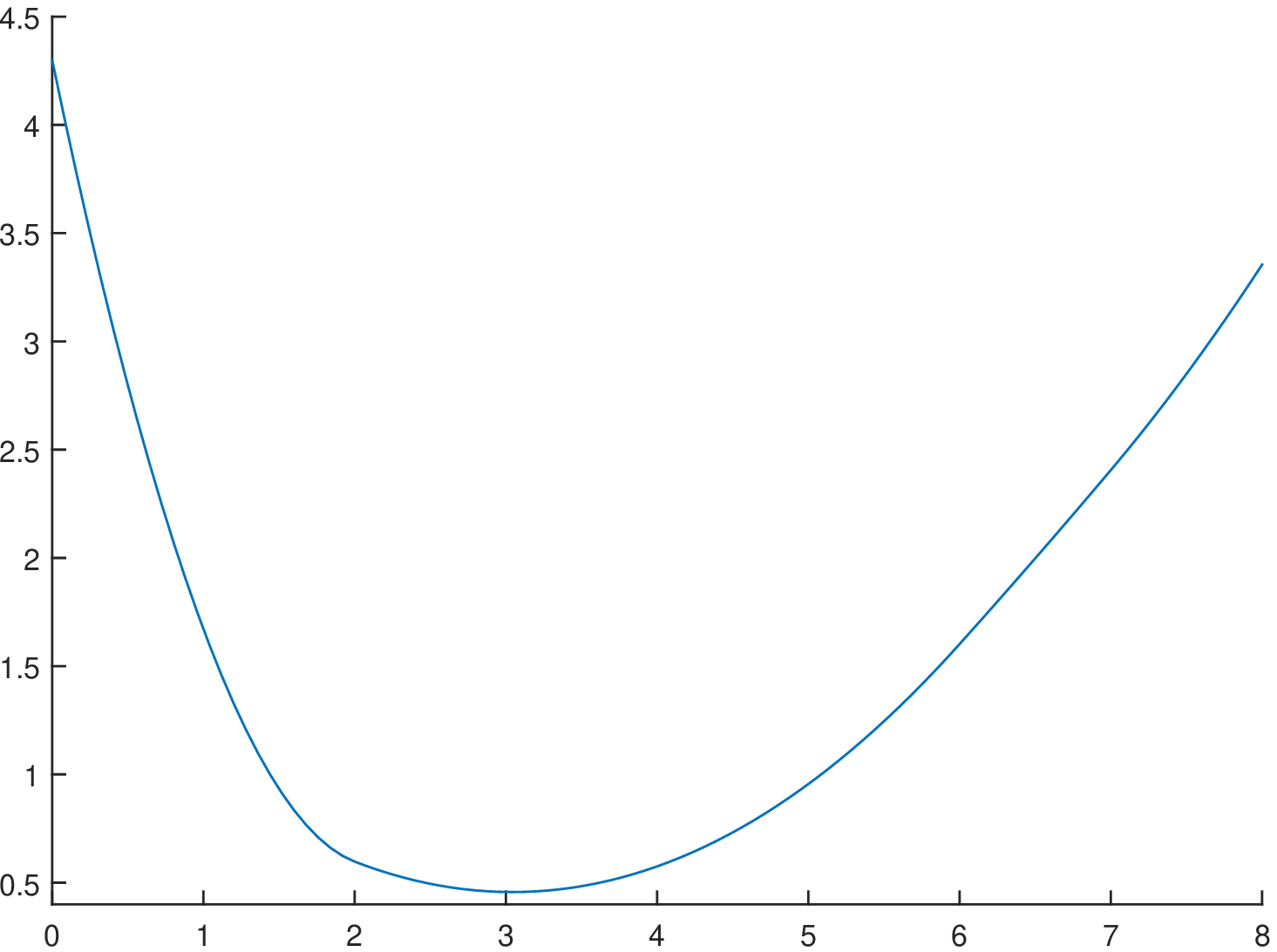} \hspace{2cm} 
 \includegraphics[width=7cm]{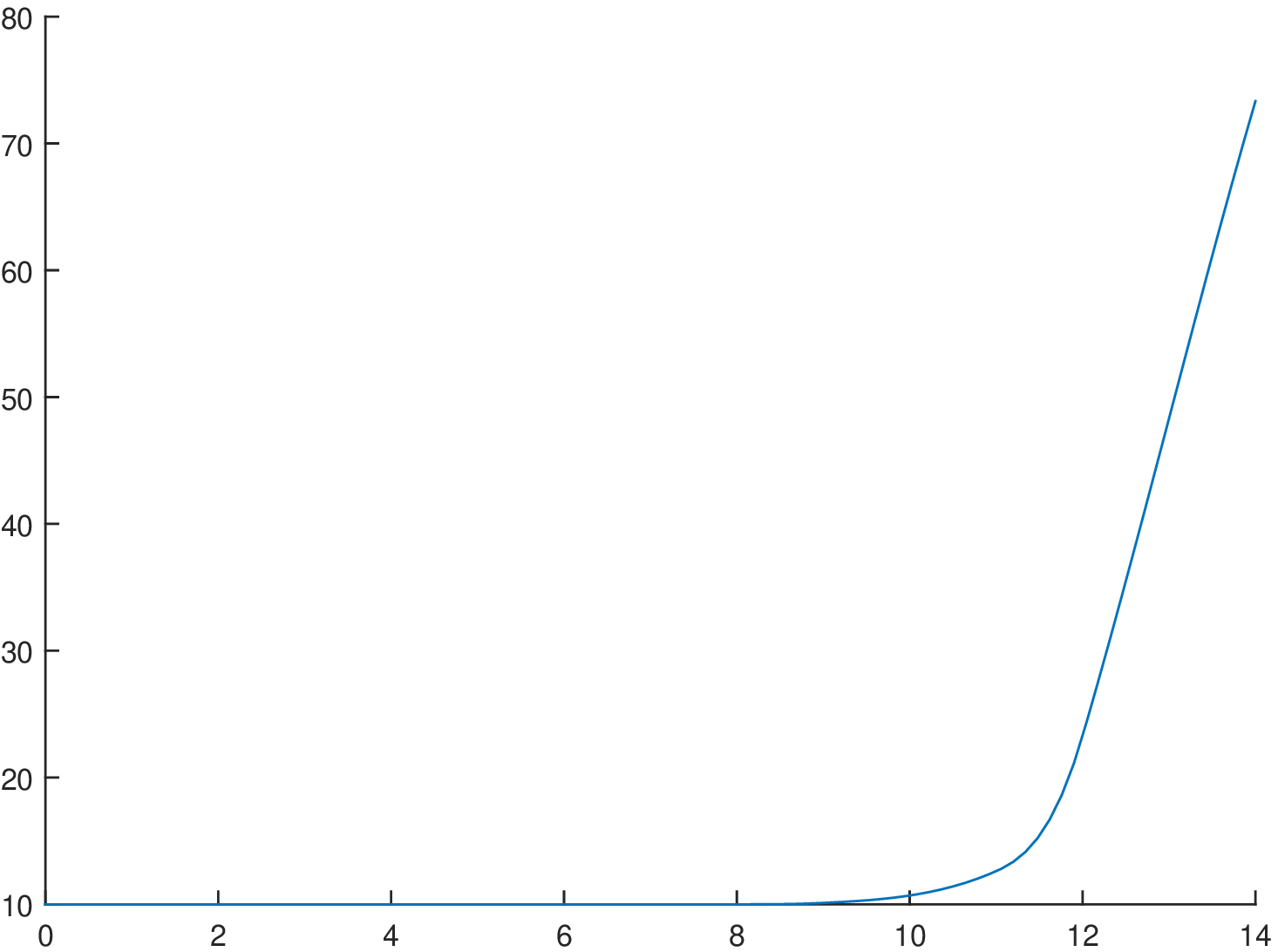} \hspace{2cm} 
\\{\footnotesize \noindent\textbf{Fig. 5.} Approximation by OCICS for Example 3. \hspace{1cm} \textbf{Fig. 6.} Approximation by OCICS for Akima's data. }\\
\vspace{0.5cm}

Example 4. The data for the last example were taken from \cite{Aki70} (see Table 1), where the conditions \eqref{zmc20_1} are not fulfilled. As for Akima's data, the solution of the system \eqref{zmcov11} does not satisfy the condition $m_{i}\leq m_{i+1}, ~i=0,\cdots,k-1$, so we can not construct the shape-preserving integro spline with the proposed method. Now, we can simply choose $m_{i}$ by
\begin{equation*}
  m_i=\delta I_i,~~i=1,\cdots,k-1,
\end{equation*}
because of \thmref{zmthm6}. The remainder $m_0$ and $m_k$ are obtained from \eqref{zmcov11} setting $i=1,k-1$, respectively. The values of $S(x_i)$ are completely determined by \eqref{zmcov7} and \eqref{zmcov8}. Fig. 6 shows that the last integro cubic spline with $\alpha=\frac12$ has a better convexity and monotonicity property.
\section*{Conclusion}
In this paper, we derive a family of $C^1$ convex integro cubic splines based on sufficient conditions for convexity. We give some sufficient convexity and monotonicity conditions for constructed integro splines. The proposed family of splines has good approximation properties. The best convex integro spline is obtained when the $\alpha$ parameter is equal to $\frac 12$. The shape-preserving properties of splines are demonstrated by numerical examples.
\bibliographystyle{unsrt}


\begin{thebibliography}{1}

	\bibitem{Aki70}
	H.~Akima, 
	\newblock A new method of interpolation and smooth curve fitting based on local procedures.
	\newblock {\em J. Assoc. Comput. Mach.},  \textbf{17} pages 589--602, 1970.

	\bibitem{Beh06}
     H.~Behforooz,
	\newblock      Approximation by integro cubic splines.
	\newblock {\em Appl. Math. Comput.}, \textbf{175} pages 8--15. 2006.

\bibitem{Fish2005}M.~Fischer and P.~Oja, 
	\newblock Monotonicity preserving rational spline histopolation.
	\newblock {\em J. Comput. Appl. Math.}, \textbf{175} pages 195--208. 2005.

\bibitem{Fish2007}M.~Fischer and P.~Oja and H.~Trossmann, 
	\newblock      Comonotone shape-preserving spline histopolation. 
	\newblock {\em J. Comput. Appl. Math.}, \textbf{200} pages 127--139. 2007.

\bibitem{Kay85}M.~Kaykobad,  
	\newblock      Positive solutions of positive linear systems. 
	\newblock {\em Linear Algebra Appl.}, \textbf{64} pages 133--140. 1985.

\bibitem{Kva12}T.~Kim and B.I. ~Kvasov, 
	\newblock      A shape-preserving approximation by weighted cubic splines.
	\newblock {\em  J. Comput. Appl. Math.}, \textbf{236} pages 4383--4397. 2012.

\bibitem{Mul1999}B.~Mulansky and J.W.~Schmidt, 
	\newblock      Convex interval interpolation using a three-term staircase algorithm.
	\newblock {\em  Numer. Math.}, \textbf{82} pages 313--337. 1999.

\bibitem{Neu78}E.~Neuman, 
	\newblock      Uniform approximation by some Hermite interpolating splines.
	\newblock {\em  J. Comput. Appl. Math.},  \textbf{4} pages 7--9. 1978.

\bibitem{Schm94}J.W.~Schmidt and W.~He{\ss},
	\newblock      Shape preserving $C^2$-spline histopolation. 
	\newblock {\em  J. Approx. Theory}, \textbf{75} pages 325--345. 1993.

\bibitem{Schm3}J.W.~Schmidt and W.~He{\ss},    
	\newblock      An allways successful method in univariate convex $C^2$-interpolation.
	\newblock {\em  Numer. Math.}, \textbf{71} pages 237--252. 1995.

\bibitem{Schm4}J.W.~Schmidt, 
	\newblock      Staircase algorithm and construction of convex spline interpolats up to the continuity $C^3$.
	\newblock {\em  Comput. Math. Appl.}, \textbf{31} pages 67--79. 1996.

\bibitem{Zhan1420} T.~Zhanlav,
	\newblock      Shape preserving properties of  $C^1$ cubic spline approximations.
	\newblock {\em   Scientific transaction NUM},  \textbf{7} pages 14--20. 2000.

\bibitem{Zhan2135}T.~Zhanlav,
	\newblock      Shape preserving properties of some $C^2$ cubic spline approximations. 
	\newblock {\em Scientific transaction NUM}, \textbf{7} pages 21--35. 2000.    

\bibitem{Zhan2017}T.~Zhanlav and R.~Mijiddorj,
	\newblock      Convexity and monotonicity properties of the local integro cubic spline.
	\newblock {\em Appl. Math. Comput.}, \textbf{293} pages 131--137. 2017.

\bibitem{Zencak2002} P.~\v{Z}en\v{c}\'{a}k,
	\newblock      The convex interpolation of histogram by polynomial splines: The existence theorem.
	\newblock {\em Acta Univ. Palacki. Olomuc, Fac rer. nat., Math.}, \textbf{41} pages 175--182. 2002.


\end{thebibliography}

\end{document}